\documentclass[12pt,a4paper]{amsart}

\usepackage{amsmath,amssymb,amsthm,mathtools}
\usepackage{geometry}
\usepackage{enumitem}
\usepackage{microtype}
\usepackage{tikz}
\usepackage[colorlinks=true,linkcolor=blue,urlcolor=blue,citecolor=blue]{hyperref}

\allowdisplaybreaks

\newtheorem{theorem}{Theorem}[section]
\newtheorem{lemma}[theorem]{Lemma}

\newtheorem{corollary}[theorem]{Corollary}
\newtheorem{problem}[theorem]{Problem}
\theoremstyle{definition}

\theoremstyle{remark}
\newtheorem{remark}[theorem]{Remark}

\newcommand{\cP}{\mathcal P}
\newcommand{\cT}{\mathcal T}

\usepackage{xcolor}
\definecolor{wxpink}{HTML}{FF1493} 

\title{Hypergraphs of Low Degree Without Independent Transversals}

\author[Xinzhe Song]{Xinzhe Song}
\author[Xia Wang]{Xia Wang\textsuperscript{*}}

\thanks{\textsuperscript{*}Both authors are supported by the Institute for Basic Science (IBS-R029-C4) and affiliated with the Extremal Combinatorics and Probability Group (ECOPRO), Institute for Basic Science (IBS), Daejeon, South Korea.  Xinzhe Song is also affiliated with the Academy of Mathematics and Systems Science, Chinese Academy of Sciences, Beijing, China, and the University of Chinese Academy of Sciences, Beijing, China. Xia Wang is also affiliated with the School of Mathematics, Shandong University, Jinan, China, and supported by the China Scholarship Council. Emails: \texttt{songxinzhe@amss.ac.cn}, \texttt{xiawang@mail.sdu.edu.cn}.}

\date{}

\begin{document}
\begin{abstract}
For every fixed $k\ge 3$ and every integer $b\ge 2$, we construct a finite vertex-partitioned $k$-uniform hypergraph with blocks of size $b$ and no independent transversal. We show that the maximum number of edges meeting a block is $\left(\frac{(k-1)^{k-1}}{k^k}+O_k(b^{-1})\right)b^k,$
while the maximum vertex degree is $\left(\frac{(k-2)^{k-2}}{(k-1)^{k-1}}+O_k(b^{-1})\right)b^{k-1}$. This shows that the known bound on the number of edges meeting a block is asymptotically sharp even when the leading constant in the maximum vertex degree tends to zero, which answers a problem of Groenland, Kaiser, Treffers, and Wales.

By blowing up the same construction, we further obtain for arbitrarily large $n$, there exists a $k$-uniform hypergraph with $n$ blocks, each of size at least $n$, and $\Delta(G)<\left(1-\frac1k\right)^{k-1}n^{k-1}$, but with no independent transversal. This gives a negative answer to a problem of Haxell and Wdowinski.
\end{abstract}

\maketitle

\section{Introduction}

Let $H$ be a hypergraph and $\cP=\{V_1,\ldots, V_m\}$ be a partition of $V(H)$, where each $V_i$ is called a block of $\cP$. A \emph{transversal} of $\cP$ is a set containing exactly one vertex from each block $V_i$. An \emph{independent transversal} of $\cP$ is a transversal $T$ such that $T$ is an independent set in $H$.
Independent transversals form a natural framework for a number of problems in extremal and probabilistic combinatorics, with connections to list coloring, linear arboricity, strong graph coloring, and matching problems; see, for instance,
\cite{Alon1988,GlockSudakov2022,Haxell2001,Haxell2004,Haxell2016,LohSudakov2007}.
Their systematic study goes back to a problem of Bollob\'as, Erd\H{o}s, and Szemer\'edi \cite{BollobasErdosSzemeredi1975}. Early work on vertex-partition transversals and strong colorings was developed by Fellows~\cite{Fellows1990} and Alon~\cite{Alon1992}, respectively. 

In the graph setting, the fundamental question is how large the blocks must be, relative to the maximum degree, in order to force an independent transversal. A classical theorem of Haxell \cite{Haxell2001} shows that every graph of maximum degree at most $d$ has an independent transversal with respect to every partition whose blocks have size at least $2d$. This is best possible: constructions of Szab\'o and Tardos \cite{SzaboTardos2006} show that blocks of size $2d-1$ do not suffice in general. For a fixed number of blocks, extremal degree thresholds were investigated by Jin~\cite{Jin1992} and Yuster~\cite{Yuster1997}, and subsequently determined in full by Haxell and Szab\'o~\cite{HaxellSzabo2006}. Substantially better bounds are possible under additional sparsity assumptions on the interaction between different blocks. For instance, Loh and Sudakov \cite{LohSudakov2007} showed that if the
local degree, that is, the maximum number of neighbors of a vertex in a single block, is $o(d)$, then blocks of size $(1+\varepsilon)d$ suffice for every fixed $\varepsilon>0$. For related results in this direction, see, for example, \cite{KangKelly2022,GlockSudakov2022,ChakrabortiTran2025}. In another direction, Haxell and Wdowinski~\cite{HaxellWdowinski2024a} obtained degree conditions combining the maximum degree and block average degree settings, also give the stability results. Explicit constructions of graphs without independent transversals have also been investigated~\cite{AharoniHolzmanHowardSprussel2015,HaxellWdowinski2024b}.

For uniform hypergraphs, however, far fewer results are known. Here it is useful to distinguish the edges that are actually relevant to a transversal. We call an edge \emph{stretched} if its vertices
lie in distinct blocks of $\cP$. Since a transversal contains at most one vertex from each block, only stretched edges can be contained in a transversal. Thus, when studying independent transversals, it is natural to measure the density using only stretched edges. In particular, two parameters arise naturally: the number of stretched edges meeting a given block and the maximum degree of a vertex.
Wanless and Wood~\cite{WanlessWood2022} proved that if $H$ is a $k$-uniform hypergraph with a partition $\cP=\{V_i\}$ such that $|V_i|\ge b$, and every block $V_i$ meets at most $\beta_k b^{k-1}|V_i|$ stretched edges, where $\beta_k:=(k-1)^{k-1}/k^k$, then $H$ has an independent transversal. Later, Groenland, Kaiser, Treffers, and Wales~\cite{GroenlandKaiserTreffersWales2023} constructed examples showing that the coefficient $\beta_k$ is asymptotically sharp.  This led them to ask whether such counterexamples can also have substantially smaller maximum vertex degree. They asked the following question.

\begin{problem}[Groenland--Kaiser--Treffers--Wales \cite{GroenlandKaiserTreffersWales2023}, Problem 13]
\label{prob:GKTW}
Is there a constant $0<\eta<1$ such that, for every $k$, every $\varepsilon>0$, and all sufficiently large integers $b$, there exists a $k$-uniform hypergraph $H$ with a vertex partition $\cP=\{V_i\}$ such that every block has size at least $b$ and meets at most
$(\beta_k+\varepsilon)b^k$ stretched edges, every vertex has degree at most
$\eta b^{k-1}$, and $H$ has no independent transversal? Can one moreover take
$\eta=\eta(k)\to0$ as $k\to\infty$?
\end{problem}

A second question, posed by Haxell and Wdowinski~\cite{HaxellWdowinski2024b}, asks whether a maximum-degree condition alone can guarantee an independent transversal when the block size is determined by the number of blocks.

\begin{problem}[Haxell--Wdowinski~\cite{HaxellWdowinski2024b}, Problem 26]\label{prob:HW}
Let $G$ be a $k$-uniform hypergraph with maximum degree $d$, and $\cP=\{V_1,\ldots,V_n\}$ be a partition of $V(G)$. Do the conditions
$|V_i|\ge n$ for every $i$ and
$d\le(1-1/k)^{k-1}n^{k-1}$ guarantee an independent transversal?
\end{problem}

We give an explicit tree construction $H_{b,k}$ (in Section \ref{sec:tree}) that applies to both problems, which is inspired by the recursive constructions of Groenland, Kaiser, Treffers, and Wales~\cite{GroenlandKaiserTreffersWales2023}. The construction of Groenland et al.\ uses complete Cartesian products in its recursive forcing mechanism, which leads to vertices of relatively large degree. Our tree-based construction serializes this forcing along root-to-leaf paths, allowing the maximum-degree coefficient to decrease to $0$ as $k\to\infty$ while preserving the optimal blockwise edge scale.

\begin{theorem}\label{thm:tree-main}
For each fixed integer $k\ge3$ and every integer $b\ge2$, there exists a finite $k$-uniform hypergraph $H_{b,k}$ with a vertex partition $\cP_{b,k}$ such that every block has size $b$, every edge is stretched, and $H_{b,k}$ has no independent transversal. Moreover, as $b\to\infty$, we have

$$\max_{V\in\cP_{b,k}} \bigl|\{e\in E(H_{b,k}):e\cap V\ne\varnothing\}\bigr| =\left(\beta_k+O_k(b^{-1})\right)b^k$$
and
$$\Delta(H_{b,k})=\left(c_k+O_k(b^{-1})\right)b^{k-1},$$
where $\beta_k:=\frac{(k-1)^{k-1}}{k^k}$ and $c_k:=\frac{(k-2)^{k-2}}{(k-1)^{k-1}}$.
\end{theorem}

Together with the case $k=2$ from~\cite{GroenlandKaiserTreffersWales2023}, Theorem~\ref{thm:tree-main} completely resolves Problem~\ref{prob:GKTW}. For $k\ge3$, the first estimate in Theorem~\ref{thm:tree-main} gives the required bound $(\beta_k+\varepsilon)b^k$ on the number of stretched edges meeting each block for sufficiently large $b$. The second estimate gives maximum degree at most $2c_kb^{k-1}$ for sufficiently large $b$, where $c_k\sim 1/(e(k-1))$ and hence $c_k\to0$ as $k\to\infty$. 

A uniform blow-up of the same construction yields a counterexample to Problem~\ref{prob:HW}, as stated below.

\begin{theorem}\label{thm:HW-counterexample}
For every fixed integer $k\ge3$, there are infinitely many integers $n$ for
which there exists a $k$-uniform hypergraph $G$ with a partition
$\cP=\{V_1,\ldots,V_n\}$ such that $|V_i|\ge n$ for every $i$, $\Delta(G)<\left(1-\frac1k\right)^{k-1}n^{k-1}$, and $G$ has no independent transversal.
\end{theorem}

The paper is organized as follows. In Section~\ref{sec:tree}, we introduce the tree construction and determine its blockwise edge count and maximum degree, yielding the constants $\beta_k$ and $c_k$. In Section~\ref{sec:blowup}, we apply a uniform blow-up to obtain Theorem~\ref{thm:HW-counterexample}.

\section{The tree construction}
\label{sec:tree}
In this section, for $k\ge3$ and $b\ge2$, we construct the hypergraph $H_{b,k}$ and verify the properties needed for Theorem~\ref{thm:tree-main}. We first define the underlying rooted tree $\mathcal T_{b,k}$, then construct $H_{b,k}$ from $\mathcal T_{b,k}$, and finally analyze its independent transversals, blockwise edge count, and maximum degree.

\subsection*{The tree $\mathcal T_{b,k}$}
We first define a family of finite rooted trees $T_b,T_{b-1},\ldots,T_0$ iteratively. Let $T_b$ consist of a root with $b$ children, all of which are leaves.
For $j=b-1,b-2,\ldots,0$, having defined $T_{j+1}$, we construct $T_j$ as follows.
The root of $T_j$ has exactly $b$ children: $j$ of them are leaves, while each of the
remaining $b-j$ children is the root of a copy of $T_{j+1}$.
Thus every internal vertex of every $T_j$ has exactly $b$ children.
Let $N_j$ denote the number of internal vertices of $T_j$.
Then $N_b=1$ and, for $0\le j<b$, $N_j=1+(b-j)N_{j+1}$. Hence $N_j=1+\sum_{q=1}^{b-j}
(b-j)(b-j-1)\cdots(b-j-q+1)$.

We now construct $\mathcal T_{b,k}$.
Start with the complete rooted $b$-ary tree of depth $k-2$.
Thus every vertex at levels $0,\ldots,k-3$ has exactly $b$ children, and all vertices
at level $k-2$ are leaves.
For each vertex $v$ at level $k-2$, attach a copy of $T_0$ by identifying $v$ with
the root of that copy.
The resulting rooted tree is denoted by $\mathcal T_{b,k}$.

\begin{figure}[ht]
    \centering
    \includegraphics[width=0.8\linewidth]{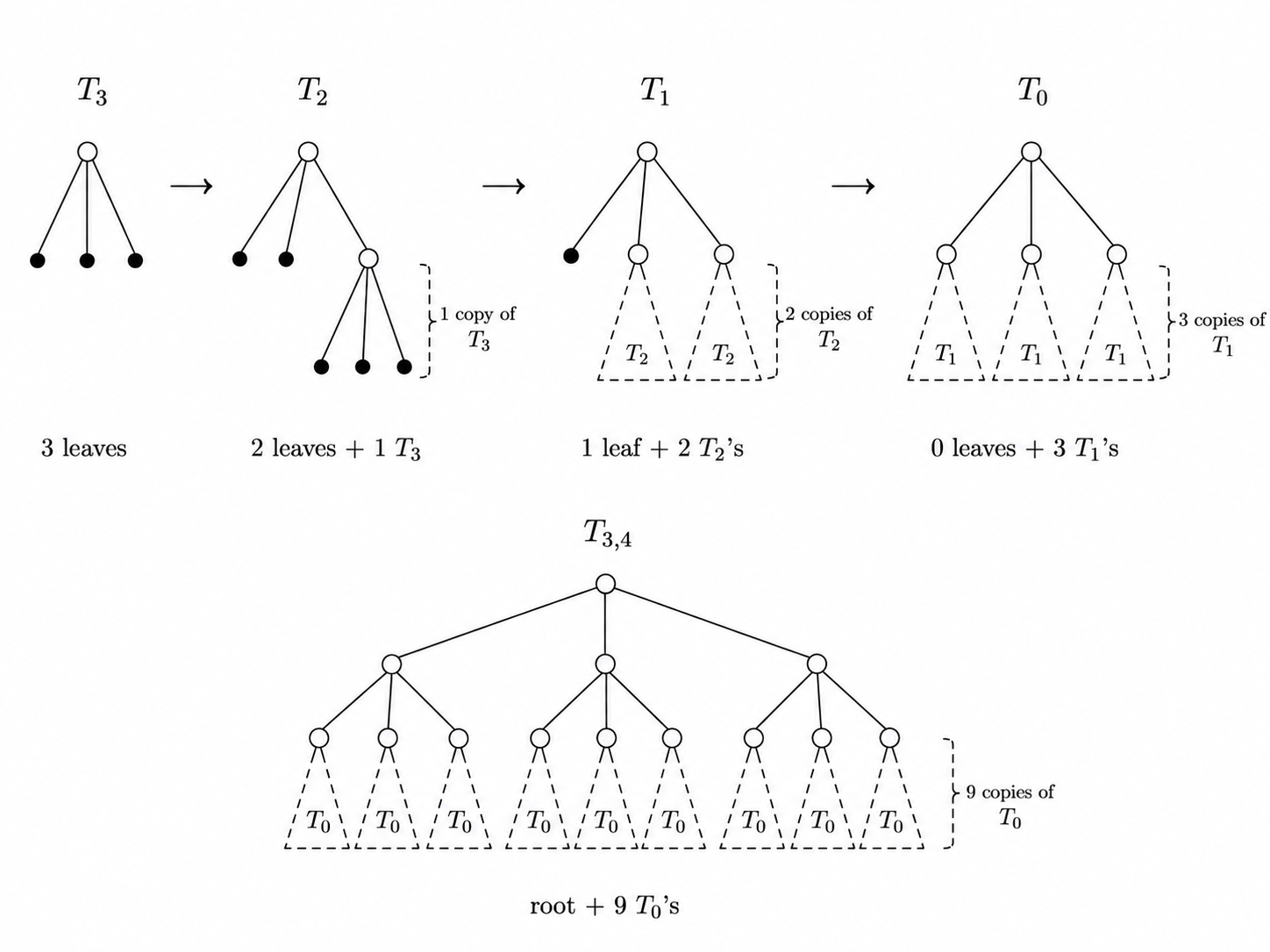}
    \caption{The tree construction for $b=3$. The upper part illustrates the recursive construction of $T_3,T_2,T_1,T_0$, while the lower part shows the resulting tree $\mathcal T_{3,4}$.}
    \label{fig:placeholder}
\end{figure} 

\subsection*{The hypergraph $H_{b,k}$}
We next construct the hypergraph $H_{b,k}$ from the rooted tree $\mathcal T_{b,k}$. For every internal vertex $u\in V(\mathcal{T}_{b,k})$, denote the $b$ edges from $u$ to its children by $t_1,\ldots,t_b$, and define $V_u:=\{(u,t_1),\ldots,(u,t_b)\}$.
The vertex set of $H_{b,k}$ is $V(H_{b,k}) = \bigcup_{\substack{u\in V(\mathcal T_{b,k})\\
u\text{ is internal}}}V_u$, and we take $\mathcal P_{b,k}:=\{V_u:u\text{ is internal in }\mathcal{T}_{b,k}\}$ as its vertex partition.

For each leaf $\ell$ of $\mathcal T_{b,k}$, consider the unique path from the root of $\mathcal T_{b,k}$ to $\ell$. Let $u_0 \xrightarrow{a_0} u_1
\xrightarrow{a_1}\cdots
\xrightarrow{a_{k-2}}u_{k-1}
\xrightarrow{a_{k-1}}u_k=\ell$ be the last $k$ edges of this path, where each $a_i=u_i u_{i+1}$ denotes the corresponding edge of the tree. Define $e_\ell:=\{(u_0,a_0),\ldots,(u_{k-1},a_{k-1})\}$.
Finally, let $E(H_{b,k}):=\{e_\ell:\ell\text{ is a leaf of }\mathcal T_{b,k}\}$.

\subsection*{Properties of $H_{b,k}$}
We first show some basic properties of $H_{b,k}$. Every block in $\mathcal P_{b,k}$ has size $b$, and every $e_\ell$ has size $k$ and meets $k$ distinct blocks. Hence $H_{b,k}$
is $k$-uniform and every edge is stretched.
Let $M_{b,k}$ denote the number of internal vertices of $\cT_{b,k}$. The vertices at levels $0,\ldots,k-3$ contribute $\sum_{i=0}^{k-3} b^i$ internal vertices and level $k-2$ consists of $b^{k-2}$ roots of copies of $T_0$. Therefore
$M_{b,k}=\sum_{i=0}^{k-3}b^i+b^{k-2}N_0$. Since $\cT_{b,k}$ has $bM_{b,k}$ edges and $M_{b,k}+L$ vertices, where $L$ is the number of leaves, it has $(b-1)M_{b,k}+1$ leaves. Consequently, $H_{b,k}$ has
$M_{b,k}$ blocks, $bM_{b,k}$ vertices, and
$(b-1)M_{b,k}+1$ hyperedges.

\begin{lemma}\label{lem:no-IT}
The hypergraph $H_{b,k}$ has no independent transversal with respect to $\cP_{b,k}$.
\end{lemma}

\begin{proof}
Let $S$ be a transversal in $H_{b,k}$. For every internal vertex $u\in V(\cT_{b,k})$, there is a unique edge $a(u)\in E(\cT_{b,k})$ such that $(u,a(u))\in S$. Start at the root of $\cT_{b,k}$, whenever the current vertex is an internal vertex $u$ of $\cT_{b,k}$, follow the edge $a(u)$ to the next level. Since the tree is finite, this process terminates at a leaf
$\ell$. By construction, the last $k$ edges of the resulting
root-to-leaf path determine a hyperedge $e_\ell \subseteq S$.
Thus every transversal contains an edge of $H_{b,k}$
and is not independent.
\end{proof}

We next consider the parameters needed for Problem~\ref{prob:GKTW} and Theorem~\ref{thm:tree-main}. We first estimate the number of hyperedges meeting each block.

\begin{lemma}\label{prop:block-asymptotic}
For every fixed $k\ge3$, we have:
\[
\max_{V\in\cP_{b,k}}
\bigl|\{e\in E(H_{b,k}):e\cap V\ne\varnothing\}\bigr|
= \left(\beta_k+O_k(b^{-1})\right)b^k,~\text{where}~\beta_k=\frac{(k-1)^{k-1}}{k^k}.
\]
\end{lemma}

\begin{proof}
Consider an internal vertex $u\in V(\cT_{b,k})$ which is the root of a copy of $T_j$. A hyperedge $e_\ell\in E(H_{b,k})$ meets $V_u$ of $\cP_{b,k}$ if and only if $\ell$ is a descendant leaf of $u$ and $q\le k$, where $q$ is the distance between $\ell$ and $u$ in $\cT_{b,k}$. 
For $q=1$, there are exactly $j$ leaves at distance one. 
For $2\le q\le b-j+1$, the number of leaves at distance exactly $q$ below the root $u$ of $T_j$ is $(j+q-1)(b-j)(b-j-1)\cdots(b-j-q+2)$. 
Indeed, to reach one such leaf, the first $q-1$ steps must successively enter roots of copies of
$T_{j+1},T_{j+2},\ldots,T_{j+q-1}$.
The numbers of choices for these first $q-1$ steps are respectively $b-j,\ b-j-1, \dots,\ b-j-q+2$.
After these $q-1$ steps, we are at the root of a copy of $T_{j+q-1}$, which has exactly $j+q-1$ leaf children.
Hence the number of such leaves is $(j+q-1)(b-j)(b-j-1)\cdots(b-j-q+2)$. It follows that the number of hyperedges meeting $V_u$ is exactly
\begin{align*}
 B_{b,k}(j):&=j+\sum_{q=2}^{k}(j+q-1)(b-j)(b-j-1)\cdots(b-j-q+2)\\
&=b\cdot\frac{j}{b}+\sum_{q=2}^k b^q
\left(\frac{j}{b}+\frac{q-1}{b}\right)
\prod_{r=0}^{q-2}\left(1-\frac{j}{b}-\frac{r}{b}\right)\\
&=\left(\frac{j}{b}(1-\frac{j}{b})^{k-1}+O_k(b^{-1})\right)b^k.   
\end{align*}

It remains to consider the blocks corresponding to vertices in the top complete $b$-ary tree of $\cT_{b,k}$. Let $u$ be a vertex at level $s$, where $0\le s\le k-3$.
Then $u$ has exactly $b^{k-2-s}$ descendants at level $k-2$, each of which is the root of a copy of $T_0$.
Fix one such root $v$, and let $\ell$ be a leaf in the corresponding copy of $T_0$.
Since $v$ lies $k-2-s$ levels below $u$, the hyperedge $e_\ell$ meets $V_u$ if and only if $\ell$ lies at most $s+2$ levels below $v$.
Hence, in each copy of $T_0$, the relevant leaves are exactly those at distance
at most $s+2$ from its root.
Therefore the number of hyperedges meeting $V_u$ is exactly
\begin{align*}
C_{b,k}(s)&:= b^{k-2-s}B_{b,s+2}(0) = b^{k-2-s}\sum_{q=2}^{s+2}(q-1)b(b-1)\cdots(b-q+2)\\
&=b^{k-2-s}\sum_{q=2}^{s+2}(q-1)b^{q-1}
\prod_{r=1}^{q-2}\left(1-\frac{r}{b}\right)\\
&=(s+1+O_k(b^{-1}))b^{k-1}.
\end{align*}

Note that
$\max_{0\le j\le b}B_{b,k}(j)
=\left(\max_{0\le j\le b}\frac{j}{b}
\left(1-\frac{j}{b}\right)^{k-1}+O_k(b^{-1})\right)b^k
=(\beta_k+O_k(b^{-1}))b^k$,
since the function $x(1-x)^{k-1}$ is maximized at $x=1/k$.
Moreover,
$\max_{0\le s\le k-3}C_{b,k}(s)
=(k-2+O_k(b^{-1}))b^{k-1}
=O_k(b^{k-1})$.
Thus we have 
\[
\max_{V\in\cP_{b,k}}
\bigl|\{e\in E(H_{b,k}):e\cap V\ne\varnothing\}\bigr|
=\max\left\{
\max_{0\le j\le b}B_{b,k}(j),
\max_{0\le s\le k-3}C_{b,k}(s)
\right\}
=(\beta_k+O_k(b^{-1}))b^k.\]
\end{proof}

Now we compute the maximum degree in $H_{b,k}$.

\begin{lemma}\label{prop:degree-asymptotic}
For every fixed $k\ge3$,
\[
\Delta(H_{b,k})
=
\left(c_k+O_k(b^{-1})\right)b^{k-1}, ~\text{where}~
c_k=\frac{(k-2)^{k-2}}{(k-1)^{k-1}}.
\]
\end{lemma}

\begin{proof}
Fix a vertex $(u,a)\in V(H_{b,k})$, where $a$ is an edge from $u$ to one of its children, say $w$. The vertex $(u,a)$ belongs to $e_\ell$ if and only if $a$ is one of the last $k$ edges of the path between the root of $\cT_{b,k}$ and $\ell$. Equivalently, $\ell$ is a descendant leaf of $w$ that can be reached from $w$ by at most $k-1$ edges.

Suppose first that $w$ is the root of a copy of $T_j$. Using the same leaf count as above, the degree of $(u,a)$ is
\begin{align*}
  A_{b,k}(j)&:=j+\sum_{q=2}^{k-1}
(j+q-1)(b-j)(b-j-1)\cdots(b-j-q+2)\\
&=\left(
\frac{j}{b}\left(1-\frac{j}{b}\right)^{k-2}
+O_k(b^{-1})
\right)b^{k-1}.  
\end{align*}

Now suppose that $w$ is a vertex at level $s$, where $1\le s\le k-3$.
There are $b^{k-2-s}$ roots of copies of $T_0$ below $w$, and the relevant
leaves are those that can be reached from these roots by at most $s+1$ edges.
Thus the degree is
\begin{align*}
P_{b,k}(s)&:=b^{k-2-s}\sum_{q=2}^{s+1}
(q-1)b(b-1)\cdots(b-q+2)\\
&=\left(s+O_k(b^{-1})\right)b^{k-2}.  
\end{align*}

Finally, if $w$ is a leaf, then $(u,a)$ has degree one. These are all
possible cases.

With the convention that the maximum over an empty set is zero, we have
\[\Delta(H_{b,k})
= \max\left\{
\max_{0\le j\le b}A_{b,k}(j),
\max_{1\le s\le k-3}P_{b,k}(s)
\right\}.\]
Since the function $x(1-x)^{k-2}$ is maximized at $x=1/(k-1)$, we have $\max_{0\le j\le b}A_{b,k}(j)=\left(c_k+O_k(b^{-1})\right)b^{k-1}$.
Moreover, $\max_{1\le s\le k-3}P_{b,k}(s)=O_k(b^{k-2})$. Therefore $\Delta(H_{b,k})=\left(c_k+O_k(b^{-1})\right)b^{k-1}$, as desired.
\end{proof}

\begin{remark}
The maximum-degree bound in Lemma~\ref{prop:degree-asymptotic} is smaller by a factor of order $k$ than the degree scale appearing in Problem~\ref{prob:HW}. Indeed, writing $\gamma_k=(1-1/k)^{k-1}$, we have
$c_k/\gamma_k=\Theta(1/k)$.
Thus, after the uniform blow-up in Section~3, the resulting counterexamples
have maximum degree smaller by a factor of order $k$ than the proposed
threshold $\gamma_k n^{k-1}$.

For $k=3$, this comparison is particularly explicit. In this case
$A_{b,3}(j)=b+bj-j^2$, and hence
$\Delta(H_{b,3})=b+\lfloor b^2/4\rfloor$.
Thus the asymptotic maximum-degree coefficient is $1/4$, whereas
$\gamma_3=4/9$.

\end{remark}
Theorem~\ref{thm:tree-main} can be proved by 
Lemma~\ref{lem:no-IT}, Lemma~\ref{prop:block-asymptotic}, and Lemma~\ref{prop:degree-asymptotic}.
\begin{corollary}\label{cor:GKTW}
Problem~\ref{prob:GKTW} has an affirmative answer. By the graph construction of Groenland et al.~\cite{GroenlandKaiserTreffersWales2023}, one may take $\eta=0.854$ for its first part. For the stronger statement, one may take $\eta(2)=0.854$ and $\eta(k)=2c_k$ for every $k\ge3$, so that $\eta(k)\to0$ as $k\to\infty$.
\end{corollary}

\section{The blow-up construction}
\label{sec:blowup}

We now construct the hypergraph $G_{b,k}$ from $H_{b,k}$ by a uniform blow-up of its vertices. Let $n:=M_{b,k}$ be the number of blocks in
$\mathcal P_{b,k}$, and set $t:=\lceil n/b\rceil$.
For each vertex $x\in V(H_{b,k})$, replace $x$ by a set $X_x:=\{x^{(1)},\ldots,x^{(t)}\}$ of $t$ copies. For each edge $e=\{x_1,\ldots,x_k\}\in E(H_{b,k})$, include in $G_{b,k}$ all edges of the form $\{x_1^{(s_1)},\ldots,x_k^{(s_k)}\}$, $(s_1,\ldots,s_k)\in[t]^k$. Thus each edge of $H_{b,k}$ gives rise to $t^k$ edges in $G_{b,k}$.
For every internal vertex $u$ of $\cT_{b,k}$, recall that $V_u=\{(u,t_1),\ldots,(u,t_b)\}$ is a block of $\mathcal P_{b,k}$.
Define the corresponding block of $G_{b,k}$ by $W_u:=\bigcup_{x\in V_u}X_x$. 
Let $\mathcal Q_{b,k}:=\{W_u:u\text{ is an internal vertex of }\cT_{b,k}\}$.
Then $\mathcal Q_{b,k}$ has the same number of blocks as $\mathcal P_{b,k}$, and each block has size $bt\ge n$.

\begin{proof}[Proof of Theorem~\ref{thm:HW-counterexample}]
The blow-up does not change the number of blocks, so $\mathcal Q_{b,k}$ has exactly $n$ blocks. Moreover, each block has size $bt\ge n$.
For every $x^{(s)}\in V(G_{b,k})$, each edge of $H_{b,k}$ containing $x$ gives rise to exactly $t^{k-1}$ edges of $G_{b,k}$ containing $x^{(s)}$. Hence $d_{G_{b,k}}(x^{(s)})=t^{k-1}d_{H_{b,k}}(x)$, and therefore
\[\Delta(G_{b,k})=t^{k-1}\Delta(H_{b,k}).\]

We next show that $G_{b,k}$ has no independent transversal. Let $S$ be
a transversal of $\mathcal Q_{b,k}$, and project each selected vertex
$x^{(s)}$ to $x$. The projected set is a transversal of
$\mathcal P_{b,k}$, and hence, by Lemma~\ref{lem:no-IT}, contains an edge
$\{x_1,\ldots,x_k\}$ of $H_{b,k}$. The corresponding selected copies of $x_1,\ldots,x_k$ form an edge of $G_{b,k}$. Thus $S$ is not independent.

It remains to verify the maximum-degree condition. Since $t=\lceil n/b\rceil$, we have $1\le bt/n<1+b/n$. Moreover, $n=M_{b,k}\ge b^{k-2}N_0$ and $N_0\ge b+1$, so $b/n=o_b(1)$.
Hence $bt/n=1+o_b(1)$. By Lemma~\ref{prop:degree-asymptotic},
\[
\frac{\Delta(G_{b,k})}{n^{k-1}}
=
\frac{\Delta(H_{b,k})}{b^{k-1}}
\left(\frac{bt}{n}\right)^{k-1}
=
c_k+o_b(1).
\]
Let $\gamma_k:=(1-1/k)^{k-1}$. Since $c_k<\gamma_k$, for all
sufficiently large $b$ we have
$\Delta(G_{b,k})<\gamma_k n^{k-1}$.
Thus $G_{b,k}$ has exactly $n$ blocks, each of size at least $n$,
satisfies the required maximum-degree bound, and has no independent
transversal.

Finally, since $M_{b,k}\to\infty$ as $b\to\infty$, this construction
gives infinitely many values of $n$.
\end{proof}

\section*{Acknowledgements}

We would like to thank Weichan Liu for suggesting this problem and for the helpful discussions and guidance that contributed to this work.

\section*{Declaration on the use of AI}

During the preparation of this work, GPT-5.6 was used to propose the core construction underlying the main results. The authors subsequently organized and formalized the construction, independently verified the mathematical arguments and the associated estimates, and identified that the construction resolves two previously posed open problems, namely Problem 13 of Groenland, Kaiser, Treffers, and Wales \cite{GroenlandKaiserTreffersWales2023} and Problem 26 of Haxell and Wdowinski \cite{HaxellWdowinski2024b}. The authors take full responsibility for the correctness and presentation of the results.

\end{document}